\documentclass[reqno,tbtags,intlimits,a4paper,oneside,11pt]{amsart}
\usepackage{amsfonts,amsmath,amssymb}
\newtheorem{Theorem}{Theorem}[section]
\newtheorem{prop}[Theorem]{Proposition}
\newtheorem{Lemma}[Theorem]{Lemma}

\newtheorem{Remark}[Theorem]{Remark}
\newtheorem{Corollary}[Theorem]{Corollary}

\def\beq#1#2\eeq{%
        \begin{equation}%
        \label{#1}%
            #2%
        \end{equation}%
    }

\usepackage{color}
\usepackage{graphicx}
\usepackage{epstopdf}

\title[Chern-Dold]{Chern-Dold character in complex cobordisms and theta divisors}

\author{V.M. Buchstaber}\address{Steklov Mathematical Institute and Moscow State University, Russia}
\email{buchstab@mi-ras.ru}

\author{A.P. Veselov}
\address{Department of Mathematical Sciences,
Loughborough University, Loughborough LE11 3TU, UK}
\email{A.P.Veselov@lboro.ac.uk}

\begin{document}

\maketitle

\begin{abstract}
We show that the smooth theta divisors of general principally polarised abelian varieties can be chosen as irreducible algebraic representatives of the coefficients of the Chern-Dold character in complex cobordisms and describe the action of the Landweber-Novikov operations on them.
We introduce a quantisation of the complex cobordism theory with the dual Landweber-Novikov algebra as the deformation parameter space and show that the Chern-Dold character can be interpreted as the composition of quantisation and dequantisation maps.
Some smooth real-analytic representatives of the cobordism classes of theta divisors are described in terms of the classical Weierstrass elliptic functions. The link with the Milnor-Hirzebruch problem about possible characteristic numbers of irreducible algebraic varieties is discussed.
\end{abstract}


\section{Introduction}

In the complex cobordism theory, going back to the foundational works of Milnor and Novikov \cite{Mil-60}, \cite{Nov-62}, a prominent role is played by the Chern-Dold character introduced by the first author in \cite{B-1970}. 

By definition, the Chern-Dold character $ch_U$  is a natural multiplicative transformation of cohomology theories
$$
ch_U: U^*(X)\to H^*(X, \Omega_U\otimes \mathbb Q),
$$
where $U^*(X)$ is the complex cobordism ring of a $CW$-complex $X$ and $\Omega_U=U^*(pt),$ where $pt$ is a point, is the cobordism ring of the stably complex manifolds (or, in short, $U$-manifolds).
It is uniquely defined by the condition that when $X=pt$ 
\beq{CDpt}
ch_U:\Omega_U \to H^*(pt, \Omega_U\otimes \mathbb Q)=\Omega_U\otimes \mathbb Q
\eeq
is the canonical homomorphism of $\Omega_U$  to $\Omega_U\otimes \mathbb Q.$ This implies that for any finite $CW$-complex $X$ the transformation 
\beq{CDX}
ch_U\otimes \mathbb Q: U^*(X)\otimes \mathbb Q\to H^*(X, \Omega_U\otimes \mathbb Q) 
\eeq
is an isomorphism of $\Omega_U$-modules.

The fundamental Milnor-Novikov result says that the coefficient ring of the theory $U^*(X)$ is the graded polynomial ring $$\Omega_U=\mathbb Z[y_1,\dots,y_n,\dots],\, \deg y_n=-2n$$
of infinitely many generators $y_n, \, n \in \mathbb N.$ 

Let $u \in U^2(\mathbb CP^\infty)$ and $z\in H^2(\mathbb CP^\infty)$ be the first Chern classes of the line bundle associated to the hyperplane section
over $\mathbb CP^\infty$ in the complex cobordisms and cohomology theory respectively.
The Chern-Dold character is uniquely defined by its action
\beq{CD}
ch_U: u \to \beta(z), \quad \beta(z):=z+\sum_{n=1}^\infty[\mathcal B^{2n}]\frac{z^{n+1}}{(n+1)!},
\eeq
where $\mathcal B^{2n}$ are certain $U$-manifolds, characterised by their properties in \cite{B-1970}. 

The series $\beta(z)$ is the exponential of  the commutative formal group $$F(u,v)=u+v+\sum_{i,j}a_{i,j}u^iv^j$$ of  the geometric complex cobordisms introduced by Novikov in \cite{Nov-67}, so that
$$
F(\beta(z),\beta(w)) = \beta(z+w).
$$
Quillen identified this group with Lazard's universal one-dimensional commutative  formal group  \cite{Quil-69}.

The inverse of this series is the logarithm of this formal group, which can be given explicitly by
the Mischenko series  \cite{Nov-67, BMN-71}:
\beq{Mis}
\beta^{-1}(u)=u+\sum_{n=1}^\infty[\mathbb CP^n]\frac{u^{n+1}}{n+1}.
\eeq
The question whether there are smooth irreducible algebraic representatives of the cobordism classes $[\mathcal B^{2n}]$ in the exponential of the formal group given by the Chern-Dold character was open for a long time since 1970 (see \cite{B-1970}).

In this paper we give the following answer to this question, presenting an explicit form of the series (\ref{CD}) as
\beq{CDnew}
\beta(z)=z+\sum_{n=1}^\infty[\Theta^n]\frac{z^{n+1}}{(n+1)!},
\eeq
where $\Theta^n$ is a smooth theta divisor of a general principally polarised abelian variety $A^{n+1}$, considered as the complex manifold of real dimension $2n.$
The cobordism class of the theta divisor does not depend on the choice of such abelian variety provided $\Theta^n$ is smooth, which is true in general case \cite{AM67}.

\begin{Theorem}
The theta divisor $\Theta^n$ of a general principally polarised abelian variety $A^{n+1}$ is a smooth irreducible projective variety, which can be taken as an algebraic representative of the  coefficient $[\mathcal B^{2n}]$ in the Chern-Dold character.
\end{Theorem}

As a corollary we have the following representation of the cobordism class of any $U$-manifold $M^{2n}$ in terms of the theta divisors:
\beq{decom}
[M^{2n}]=\sum_{\lambda: |\lambda|=n}c^{\nu}_\lambda(M^{2n})\frac{[\Theta^\lambda]}{(\lambda+1)!},
\eeq
where the sum is over all partitions $\lambda=(i_1,\dots, i_k)$ of $|\lambda|=i_1+\dots +i_k=n,$ 
\beq{prodt}
\Theta^\lambda:=\Theta^{i_1}\times\dots \times\Theta^{i_k},
\eeq
$(\lambda+1)!:=(i_1+1)!\dots(i_k+1)!$ and $c^{\nu}_\lambda(M^{2n})\in \mathbb Z$ are the Chern numbers of $M^{2n}$ corresponding to the normal bundle $\nu(M^{2n})$ (see the next section for details).

As another corollary we have the following explicit expression of the exponential generating function of any Hirzebruch genus $\Phi$ of theta divisors:
\beq{hirzeb}
\Phi(\Theta, z):=\sum_{n=1}^\infty\Phi(\Theta^n)\frac{z^{n+1}}{(n+1)!} = \frac{z}{Q(z)},
\eeq
where $Q(z)=1+\sum_{n\in \mathbb N}a_nz^n$ is the characteristic power series of Hirzebruch genus $\Phi$ (see Section 3 for details).

In particular, for the Todd genus we have $Q(z)=\frac{z}{1-e^{-z}}$, so
$$Td(\Theta, z):=\sum_{n=1}^\infty Td(\Theta^n)\frac{z^{n+1}}{(n+1)!} = 1-e^{-z}=\sum_{n\in \mathbb N}(-1)^n\frac{z^{n+1}}{(n+1)!},$$
so that the Todd genus of the theta divisors is
\beq{todd1}
Td (\Theta^n)=(-1)^n
\eeq
 (cf. \cite{B-1970}).
Thus we have the following formula for the Todd genus for any $U$-manifold $M^{2n}$
\beq{Todd}
Td(M^{2n})=\sum_{\lambda: |\lambda|=n}c^{\nu}_\lambda(M^{2n})\frac{(-1)^n}{(\lambda+1)!}.
\eeq
Since the Todd genus is integer, this implies the divisibility condition on the Chern numbers $c^{\nu}_\lambda(M^{2n})$ of $U$-manifolds. As it follows from the results of Stong \cite{Stong-65} and Hattori  \cite{Hat-66}, all divisibility conditions for $U$-manifolds one can get applying to formula (\ref{decom}) the Landweber-Novikov operations and taking the Todd genus (see the discussion in the last section).

The action of the Landweber-Novikov operations on the theta divisors can be described explicitly in the following way.

Let $\lambda=(i_1,\dots, i_k)$ be a partition of $|\lambda|:=i_1+\dots +i_k$ with $(k)=(k,0,\dots,0)$ being a one-part partition. Let $S_\lambda[M]$ be the result of the action of the Landweber-Novikov operation $S_\lambda$ on $U$-manifold $M$ defined in terms of its stable normal bundle (see \cite{L-1967, Nov-67}). 
Consider a smooth complete intersection 
\beq{inters}
\Theta_{k}^{n-k}=\Theta^n\cap \Theta^n(a_1) \cap \dots \Theta^n(a_k)
\eeq
of $\Theta^n$ with $k$ general translates $\Theta^n(a_i), \, a_i \in A^{n+1}$ of the theta divisor $\Theta^n,$ which for $k<n$ is an irreducible algebraic variety (see Section 3 below).

Let $D=c_1(L)\in H^2(A^{n+1}, \mathbb Z)$ be the first Chern class of the principal polarisation bundle $L$, which is the cohomology class Poincar\'e dual to the cycle defined by $\Theta^n \subset A^{n+1}$. Then $\Theta_k^{n-k}$ is a realisation of the homology class from $H_{2n-2k}(A^{n+1}, \mathbb Z)$, which is Poincar\'e dual to $D^{k+1} \in H^{2k+2}(A^{n+1}, \mathbb Z).$ 

Note that for a general principally polarised abelian variety the cohomology class $D^p$ generates the corresponding Hodge group $H^{2p}_{Hodge}(A^{n+1})=H^{2p}(A^{n+1},\mathbb Q)\cap H^{p,p}(A^{n+1})$ for all $p$, due to Mattuck \cite{mat}, who proved the Hodge $(p,p)$-conjecture in this case (see section 17.4 in \cite{BL}).

\begin{Theorem}
Let $\lambda$ be a partition with $|\lambda| < n$ and $S_\lambda$ be the corresponding Landweber-Novikov operation, then the cobordism class $S_\lambda[\mathcal B^{2n}]$ has a smooth irreducible algebraic representative.
More precisely, if $\lambda$ is not a one-part partition, then $S_\lambda[\mathcal B^{2n}]=S_\lambda[\Theta^n]=0,$ while for $\lambda=(k), \, k\leq n$ we have
\beq{LN}
S_{(k)}[\mathcal B^{2n}]=S_{(k)}[\Theta^n]=[\Theta_{k}^{n-k}].
\eeq
\end{Theorem}

As a corollary we have the following expression for $[\Theta_{k}^{n-k}]$ as a residue at zero
\beq{res}
[\Theta_{k}^{n-k}]=\frac{(n+1)!}{2\pi i} \oint \beta(z)^{k+1} \frac{dz}{z^{n+2}},
\eeq
with $\beta(z)$ given by (\ref{CDnew}). Moreover, the cobordism class $[\Theta_{k}^{n-k}]$ is a polynomial of $[\Theta^1], \dots, [\Theta^{n-k}]$ with positive integer coefficients, which implies that the polynomial subring $\Theta_U \subset \Omega_U$ generated by the theta-divisors:
\beq{thetau}
\Theta_U=\mathbb Z[t_1,t_2,\dots], \quad t_k=[\Theta^k], \, k \in \mathbb N
\eeq
is invariant under the Landweber-Novikov operations (see section 4 below).

We use the Landweber-Novikov algebra $S$ over $\mathbb Q$, which is a graded Hopf algebra generated as vector space by $S_\lambda$, to define the  {\it quantum complex cobordism theory} as the extraordinary cohomology theory $$QU^*=U^*\otimes S^*$$ with $QU^*(pt)=Q\Omega^*:=\Omega_U\otimes   S^*$, where $S^*=Hom(S, \mathbb Q)$ is the dual  Landweber-Novikov algebra, considered here as the deformation parameter space. The cohomology theory $QU^*$ admits a geometric realisation as double cobordism theory $DU^*$, which was introduced in \cite{B-1995} in relation with Drinfeld's quantum double and studied in more detail in \cite{BR}.

From the results of Landweber \cite{L-1967} and Novikov \cite{Nov-67} (see also \cite{BSh}) it follows that there is a canonical isomorphism of algebras 
$
\sigma: S^*\cong \Omega_U\otimes \mathbb Q.
$
We show that the image of the dual basis $S^\lambda \in S^*$ can be given explicitly as
\beq{so}
\sigma(S^\lambda)=\frac{[\Theta^\lambda]}{(\lambda+1)!},
\eeq
where $\Theta^\lambda$ are the products of theta divisors (\ref{prodt}) with the cobordism classes $[\Theta^\lambda]$ giving the canonical basis in our ring $\Theta_U.$

Inspired by constructions from \cite{B-1995}), we introduce the quantisation map 
$
q^\star: U^*(X)\to QU^*(X)=U^*(X)\otimes S^*
$
as
\beq{quant}
q^\star(x)=x\otimes 1+\sum_{\lambda}S_\lambda(x)\otimes S^\lambda \in QU^*(X), \quad x \in U^*(X),
\eeq
where the sum here is over all non-empty partitions.

Define also a dequantisation map
\beq{dequant}
\mu^\star: U^*(X)\otimes S^* \to H^*(X, \Omega_U \otimes \mathbb Q), \quad \mu^\star=\mu\otimes \sigma,
\eeq
where $\mu: U^*(X) \to H^*(X,\mathbb Z)$ is the cycle realisation homomorphism, which is defined uniquely by the property 
$
\mu(u)=z
$ 
for the same $u \in U^2(\mathbb CP^\infty)$ and $z\in H^2(\mathbb CP^\infty, \mathbb Z)$ as before.

This allows us to interpret Chern-Dold character in complex cobordisms as the composition
$$
U^*(X)\stackrel{q^\star}{\longrightarrow}U^*(X)\otimes S^* \stackrel{\mu^\star}{\longrightarrow} H^*(X, \Omega_U \otimes \mathbb Q).
$$

\begin{Theorem}
The Chern-Dold character in the complex cobordisms is the composition of quantisation and dequantisation maps:
\beq{chmuint}
ch_U=\mu^\star \circ q^\star.
\eeq
\end{Theorem}

For $X=pt$ this composition is the canonical embedding $\Omega_U \to \Omega_U\otimes \mathbb Q$, but even in this case this leads to a non-trivial formula (\ref{decom}) (see Section 4).

We should mention here very interesting work by Coates and Givental \cite{CG, CG2}, who considered an analogue of quantum cohomology with the corresponding Gromov-Witten invariants \cite{Witten} taking values in complex cobordisms. Some important relations of complex cobordisms with conformal field theory and integrable systems were discussed also in \cite{KSU, Krichever, Morava}.

We consider also the most degenerate case of abelian variety $A^{n+1}=\mathcal E^{n+1}$, where $\mathcal E$ is an elliptic curve. In that case the theta-divisor is singular, but we show that there is a smooth real-analytic representative of the same homology class in $H_{2n}(A^{n+1}, \mathbb Z)$. More precisely, we have the following result (see more details in Section 5).

\begin{Theorem}
There is a smooth real-analytic $U$-manifold $\mathcal M_W^{2n} \subset \mathcal E^{n+1}$ given in terms of classical Weierstrass functions, which can be used as a representative of the cobordism class $[\Theta^{n}]$.

For every $k>1$ the cobordism class $k^{n+1}[\Theta^{n}]$ can be realised by an irreducible algebraic subvariety of $\mathcal E^{n+1}.$
\end{Theorem}

The structure of the paper is following. We start with a review of the main notions and results in complex cobordism theory, including the Chern-Dold character and Riemann-Roch-Grothendieck-Hirzebruch theorem. 

In the central section 3 we describe the topological properties of the smooth theta divisors and use them to express explicitly the Todd class and Chern-Dold character in complex cobordism theory. We introduce also the corresponding dual complex bordism classes and study their properties. 

In section 4 we discuss the Landweber-Novikov algebra and use its dual algebra as the deformation parameter space for certain quantisation of the complex cobordism theory, giving a different interpretation of the Chern-Dold character. We describe the action of the Landweber-Novikov operations on the cobordism classes of the theta divisors in terms of the algebraic cycles in general abelian varieties.

In section 5 we present some real-analytic representatives of the cobordism classes of theta divisors written explicitly in terms of the classical Weierstrass sigma and zeta functions. 

In the last section we discuss the link with the Milnor-Hirzebruch problem about description of possible characteristic numbers of the smooth irreducible algebraic varieties.

\section{Complex bordisms and cobordisms}

We present now a brief review of the complex cobordism theory, referring for the details to Stong's lecture notes \cite{Stong-68}, or, for more algebraic view, to Quillen's work \cite{Quil-71}. For the theory of the characteristic classes in cohomology theory we refer to Milnor and Stasheff \cite{MS}, in $K$-theory - to Atiyah \cite{Atiyah} and in cobordism theory - to Conner and Floyd \cite{CF,CF2}, for relations with the theory of algebraic cobordisms \cite{LM} - to Panin et al \cite{PPR} and survey \cite{B-2012}.

Let $M^m$ be a smooth closed real oriented manifold. By {\it stable complex structure} (or, simply $U$-{\it structure}) on $M^m$ we mean an isomorphism of the real oriented vector bundles
$TM^m\oplus (2N-m)_{\mathbb R}\cong r\xi,$
where $TM^m$ is the tangent bundle of $M^m$, $(2N-m)_{\mathbb R}$ is trivial naturally oriented real $(2N-m)$-dimensional bundle over $M^m$, $\xi$ is a complex vector bundle over $M^m$ and $r\xi$ is its real form. A manifold $M^m$ with a chosen $U$-structure is called $U$-{\it manifold}.
 Note that a complex structure on $\xi$ 
 determines complex structure in the stable normal bundle $\nu M^m.$
 
 Two closed smooth real oriented $m$-dimensional $U$-manifolds $M_1$ and $M_2$ are called {\it $U$-cobordant} if there exists a real $(m+1)$-dimensional $U$-manifold $W$ with boundary such that the boundary $\partial W$ is a disjoint union of $M_1^m$ with given orientation and $M_2^m$ with the opposite orientation, and such that the restriction of the stable complex normal bundle $\nu W$ to $M_i$ coincides with the stable complex normal bundles $\nu M_i, \, i=1,2.$ The notion of bordisms of $U$-manifolds is a bit more involved, see details in \cite{CF2}, Ch. 1. 
 
 

 Define the following operations in the corresponding equivalence classes of $U$-manifolds.
 The sum of bordism classes of two closed $U$-manifolds $M_1^m$ and $M_2^m$ is defined as
 $
 [M^m_1]+[M^m_2]=[M^m_1\cup M^m_2],
 $
where $M^m_1\cup M^m_2$ is the disjoint union of $M_1^m$ and $M_2^m$.
 Similarly define the product of the bordism classes of $M_1^{m_1}$ and $M_2^{m_2}$ by
 $
 [M_1^{m_1}] [M_2^{m_2}]=  [M_1^{m_1}\times M_2^{m_2}],
 $
 where $M_1^{m_1}\times M_2^{m_2}$ is the corresponding direct product. 
 This defines the commutative graded ring $\Omega^U=\sum_{m\geq 0}\Omega_m^U,$ where $\Omega_m^U$ is the group of bordism classes of $m$-dimensional $U$-manifolds. 
 Similarly, we have the graded ring $\Omega_U=\sum_{m\geq 0}\Omega^{-m}_U$, where $\Omega^{-m}_U$ is the group of cobordism classes of $m$-dimensional $U$-manifolds. 
 
From the correspondence between stable complex structures in tangent and normal bundles it follows that the groups $\Omega_m^U$ and $\Omega^{-m}_U$ and the rings $\Omega_U$ and  $\Omega^U$ are isomorphic. This isomorphism can be extended to Poincare duality between complex bordisms and cobordisms for any $U$-manifold.

Let $\lambda=(i_1, \dots, i_k), \, i_1 \geq \dots \geq i_k$ be a partition of $n=i_1+\dots+i_k$, and $p(n)$ be the number of such partitions.
Using the standard splitting principle  one can define the Chern classes $c_\lambda(TM)\in H^{2n}(M,\mathbb Z)$ of a $U$-manifold $M$ corresponding to the monomial symmetric functions $m_\lambda(t)=t_1^{i_1}\dots t_k^{i_k}+\dots$ (see \cite{MS}). We would like to emphasize the choice of monomial symmetric functions here, which is essential for us.

The {\it Chern number} $c_{\lambda}(M^{2n}), \, |\lambda|=n$ of $U$-manifold $M^{2n}$ is defined as the value of the cohomology class $c_\lambda(TM^{2n})$ on the fundamental cycle $\langle M^{2n} \rangle$:
\beq{cht}
c_{\lambda}(M^{2n}):= (c_\lambda(TM^{2n}), \langle M^{2n} \rangle).
\eeq
We have $p(n)$ Chern numbers $c_{\lambda}(M^{2n})$, which depend only on the bordism class of $M^{2n}.$

The following fundamental result is due to Milnor and Novikov.

\begin{Theorem} (Milnor \cite{Mil-60}, Novikov \cite{Nov-62})
The graded complex bordism ring $\Omega^U$ is isomorphic to the graded polynomial ring $\mathbb Z[y_1, \dots, y_{n},\dots]$ of infinitely many variables $y_{n}, n \in \mathbb N$, where $\deg y_{n}=2n.$
In particular, $\Omega_{2n-1}^U=0.$

Two closed $2n$-dimensional $U$-manifolds $M_1$ and $M_2$ are $U$-bordant if and only if all the corresponding Chern numbers are the same.
\end{Theorem}


The choice of suitable algebraic representatives of the bordism classes $y_{k} \in \Omega^U, \, k \in \mathbb N$ was discussed starting from the work of Milnor and Novikov, see the references and latest results in \cite{SU}.

It will be more convenient for us to use the Chern numbers $c^\nu_{\lambda}(M^{2n})$ defined using the stable normal bundle $\nu M^{2n}$:
\beq{chn}
c^\nu_{\lambda}(M^{2n}):= (c_\lambda(\nu M^{2n}), \langle M^{2n}\rangle).
\eeq
They can be expressed through the usual Chern numbers $c_{\lambda}(M^{2n})$ and contain the same information about $U$-manifold $M^{2n}.$

We will use the following convenient class of $U$-manifolds from \cite{B-2020}.
 
 Let $M^{2n}$ be a smooth real manifold of dimension $2n$. A {\it complex framing} of $M^{2n}$ is a choice of complex line bundle $\mathcal L$ on $M^{2n}$, such that the direct sum $TM^{2n}\oplus \mathcal L$ admits a structure of trivial complex vector bundle. Thus complex framing is a $U$-structure of very special type.
The examples of such structures is given by the following natural construction.

Let $X$ be a complex manifold of (complex) dimension $n+1$ with holomorphically trivial tangent bundle and $L$ be a complex line bundle over $X.$ Let $S$ be a real-analytic section $S: X \to L$, transversal to the zero section and consider $M^{2n}=\{x \in X: S(x)=0\} \subset X$, which is a smooth real-analytic submanifold of $X$.
Then the line bundle $\mathcal L=i^*(L)$, where $i: M^{2n} \to X$ is the natural embedding, determines the complex framing on $M.$

In our main example $X$ is a principally polarised abelian variety and $L$ is the canonical line bundle, with holomorphic section given by the $\theta$-function (see next section).
A more explicit example of a real-analytic submanifold for $X$ being a product of elliptic curves is discussed in section 5.

In complex cobordism theory there exists an analogue of the celebrated Riemann-Roch formula \cite{B-1970}.
To explain it recall first its Hirzebruch version. 

Let  $X$ be a $CW$-complex and $\xi \to X$ be a complex vector bundle over $X$.
The {\it characteristic Todd class} $Td(\xi) \in H^*(X, \mathbb Q)$ of $\xi$ is uniquely defined by following properties:

\begin{itemize}
\item $Td(\xi_1 \oplus \xi_2)= Td(\xi_1) Td(\xi_2)$;

\item $Td(\eta)=\frac{z}{1-\exp(-z)}$, where $\eta$ is the line bundle associated to the hyperplane section over $\mathbb{C}P^n$
and $z = c_1(\eta) \in H^2(\mathbb{C}P^n; \mathbb Z)$ for every $n$.
\end{itemize}
The Todd class $Td: K(X) \to H^*(X, \mathbb Q)$ in $K$-theory and the classical Chern character $ch: K(X) \to H^*(X, \mathbb Q)$ are related by the fundamental relation
\beq{toddch}
Td(\eta)ch(c_1^K(\eta))=c_1(\eta)
\eeq
where $c_1^K(\eta) =1-\eta^*, \,\, \eta^*=Hom(\eta, \mathbb C)$ is the first Chern class of the line bundle $\eta$ in $K$-theory (see \cite{B-1970, CF}).

The {\it Todd genus} of a $U$-manifold $M^{2n}$ is the characteristic number $$Td(M^{2n})=(Td(TM^{2n}),\langle M^{2n}\rangle),$$ where $\langle M^{2n}\rangle$ is the fundamental cycle of manifold $M^{2n}$. Todd genus defines the ring homomorphism $\Omega_U \to \mathbb{Z}$, which, due to Thom's results \cite{Thom-54}, is uniquely determined by the condition that
$Td(\mathbb{C}P^n) =1$ for all $n$.


Let $M$ be a smooth complex algebraic variety and $g_i$ be the complex dimension of the space of holomorphic forms of degree $i$ on $M$. Then Todd genus $Td(M)$ of $M$ is equal to its arithmetic genus, or holomorphic Euler characteristic:
$Td(M) = \sum\limits_{i\geqslant0}(-1)^i g_i$ (see \cite{Hirz, Hirz-66}).

More generally, let  $\xi$ be a holomorphic vector bundle over $M$, $ch(\xi)$ be its Chern character \cite{MS} and
$H^i(M, \mathcal O(\xi))$ be the corresponding cohomology groups \cite{Hirz-66}, then we have the following generalisations of the Riemann-Roch formula due to Hirzebruch and Grothendieck.

\begin{Theorem}(Riemann-Roch-Hirzebruch \cite{Hirz, Hirz-66})
\[
\sum_{i\geqslant0}(-1)^i \dim_\mathbb{C} H^i(M, \mathcal O(\xi)) = (ch (\xi) \, Td(TM),\langle M\rangle).
\]
\end{Theorem}

\begin{Theorem}(Riemann-Roch-Grothendieck \cite{Gro, ful})

Let $f \colon X\to Y$ be a proper morphism of smooth algebraic varieties. Then for any $x\in K(X)$
we have 
\beq{RRG}
ch(f_!x)Td(TY) = f_*(ch (x)\, Td(TX)),
\eeq
where $f_! \colon K(X) \to K(Y)$ and $f_* \colon H^*(X) \to H^*(Y)$ are pushforward (Gysin) homomorphisms in $K$-theory and cohomology respectively.
\end{Theorem}

The Grothendieck's version reduces to Hirzebruch's one when $Y$ is a point (see e.g. \cite{ful}). 
To describe its extension in the theory of complex cobordisms we recall that the characteristic Todd class $Td_U(\xi)$ of complex vector bundle over $CW$-complex $X$ with values in $H^*(X, \Omega_U\otimes \mathbb Q)$ is uniquely defined by the following properties (see \cite{B-1970}):

\begin{itemize}
\item For every two vector bundles $\xi_1$ and $\xi_2$ over $X$
$$Td_U(\xi_1 \oplus \xi_2)= Td_U(\xi_1) Td_U(\xi_2),$$

\item For any $U$-manifold $M^{2n}$
$$
(Td_U(TM^{2n}),\langle M^{2n}\rangle)=[M^{2n}],
$$
where $TM^{2n}$ is the tangent bundle of $M^{2n}$ and $[M^{2n}]$ is its bordism class.
\end{itemize}

Note that from the first condition the Todd class is uniquely defined by its value on the line bundle $\eta$ over $\mathbb{C}P^N$:
\beq{tda}
Td_U(\eta)=1+\sum_{n \in \mathbb N}A_nz^n, \quad z=c_1(\eta) \in H^2(\mathbb{C}P^N,\mathbb Z),
\eeq
where the coefficients $A_n\in \Omega^{-2n}_U\otimes \mathbb Q$ are determined by the second condition.
One of the results of this paper is the formula 
\beq{tdb}
Td_U(\eta)=\frac{z}{\beta(z)},
\eeq
where $\beta(z)$ is expressed in terms of the theta divisors by (\ref{CDnew}).


For any complex vector bundle $\xi$ over $CW$-complex $X$ we have
\beq{tdxi}
Td_U(\xi)=1+\sum_{\lambda}c_\lambda(\xi)A_\lambda,
\eeq
where the sum is over all partitions $\lambda=(i_1,\dots,i_k)$, $c_\lambda(\xi)\in H^{2|\lambda|}(X, \mathbb Z)$ are the corresponding Chern classes and 
$A_\lambda=A_{i_1}\dots A_{i_k}.$

The analogue of the fundamental relation (\ref{toddch}) has the form 
\beq{tdchcob}
Td_U(\eta)ch_U(c_1^U(\eta))=c_1(\eta)
\eeq
where the cobordism class $c_1^U(\eta) \in U^2(\mathbb{C}P^n)$ is the first Chern class of the canonical bundle $\eta$ in complex cobordisms \cite{CF}, which is dual to the bordism class of the canonical embedding $\mathbb CP^{n-1} \subset \mathbb CP^n$. This relation is a key ingredient in the proof of the following analogue of Riemann-Roch formula in complex cobordisms \cite{B-1970}.

For $U$-manifolds we have the Poincar\'e-Atiyah duality in complex cobordisms \cite{Stong-68}:
\[
D_U \colon U_k(M^{2n}) \to U^{2n-k}(M^{2n}), \qquad D^U \colon U^k(M^{2n}) \to U_{2n-k}(M^{2n}).
\]
Let $f \colon M_1^{2n} \to M_2^{2m}$ be the mapping of two $U$-manifolds, $f_*: U_k(M^{2n}) \to U_k(M^{2m})$ be the standard bordism homomorphism, and 
\beq{gysin}
f_\sharp = (D_2)_U f_*(D_1)^U \colon U^k(M_1^{2n})\to U^{2m-2n+k}(M_2^{2m})
\eeq
be the corresponding Gysin homomorphism in complex cobordisms \cite{B-1970}. 
In particular, when  $a\in U^{2k}(M^{2n})$ is the cobordism class dual to the bordism class of a smooth $U$-submanifold $M^{2n-2k}\subset M^{2n}$, then for the mapping
$f \colon M^{2n} \to M^0=pt$ we have $f_\sharp a = [M^{2n-2k}]$.


\begin{Theorem} (Riemann-Roch-Grothendieck-Hirzebruch \cite{B-1970})

Let $f \colon X\to Y$ be a mapping of closed $U$-manifolds, then for any $x\in U^k(X)$
we have 
\beq{RRGH}
ch_U(f_\sharp x) Td_U(TY) = f_!(ch_U (x) Td_U(TX)),
\eeq
where $f_!\colon H^*(X,\mathbb Q) \to H^*(Y, \mathbb Q)$ is the Gysin homomorphism.
\end{Theorem}

In particular, when $Y=pt$ is a point, we have $ch_U(f_\sharp x)=f_\sharp x$ by (\ref{CDpt}), $Td_U(TY) =1$, so in the left hand side we have simply $f_\sharp x$.
Since for the mapping $f: X \to pt$ we have $f_!(a)=(a,\langle X \rangle)$, the formula 
(\ref{RRGH}) reduces to
\beq{RRGHpt}
f_\sharp x = (ch_U (x) Td_U(TX), \langle X \rangle).
\eeq

So we see that both Chern-Dold character $ch_U$ and Todd class $Td_U$ play a prominent role in these formulas. 

We will show now that they both can be explicitly described in terms of the theta divisors (see Theorems 3.1 and 3.5 below).








\section{Theta divisors, Todd class and Chern-Dold character in complex cobordisms}

%

Let $A^{n+1}=\mathbb C^{n+1}/\Gamma$ be a principally polarised abelian variety with lattice $\Gamma$ generated by the columns of the $(n+1)\times 2(n+1)$ matrix $(I, \,\, \tau)$ with complex symmetric $(n+1)\times (n+1)$ matrix $\tau$ having positive imaginary part \cite{GH}. Its polarisation line bundle $L$ has one-dimensional space of sections generated by the classical Riemann $\theta$-function
\beq{theta}
\theta(z, \tau)=\sum_{l\in \mathbb Z^{n+1}}\exp [\pi i (l, \tau l) + 2\pi i (l, z)], \, z \in \mathbb C^{n+1}.
\eeq
The corresponding theta divisor $\Theta^n \subset A^{n+1}$ given by $\theta(z,\tau)=0$ is known (after Andreotti and Mayer \cite{AM67}) to be smooth for a general principally polarised abelian variety $A^{n+1}.$   The topology of the smooth theta divisor does not depend on the choice of such abelian variety.

In particular, for $n=1$ a generic abelian surface $A^2$ is the Jacobi variety of a smooth genus 2 curve $\mathcal C$ with theta divisor $\Theta^1 \cong \mathcal C$. For $n=2$ an indecomposable $A^3$ is Jacobi variety of a smooth genus 3 curve $\mathcal C$; in that case $\Theta^2 \cong S^2(\mathcal C)$ is smooth for all non-hyperelliptic curves $\mathcal C$, which must be then trigonal. For $n\geq 3$ the general case of $A^{n+1}$  is not Jacobian, and the theta divisor is smooth outside a locus in the moduli space of the abelian varieties of complex codimension 1. For more detail on the geometry of theta divisors we refer to the survey \cite{GruH} by Grushevsky and Hulek.

The line bundle $L$ is ample with $L^{\otimes 3}$ known (after Lefschetz \cite{BL}) to be very ample, so that the sections of $L^{\otimes 3}$ determine the embedding of $A^{n+1}$ into corresponding projective space $\mathbb P^N, \, N=3^{n+1}-1.$
The corresponding quadratic and cubic equations, defining the image in $\mathbb P^N$,  were described by Birkenhake and Lange \cite{BL1990} (see also Ch. 7 in \cite{BL}).
For the elliptic curves this reduces to the Hasse cubic equation
$
x^3+y^3+z^3=3\lambda xyz.
$

Note that the line bundle $L^{\otimes 2}$ is not very ample; its sections define the quotient $A^{n+1}/\mathbb Z_2$ by the involution $z \to -z$, which is known as {\it Kummer variety.}
When $n=1$ this is the famous Kummer quartic surface in $\mathbb P^3$ with 16 singular points, see e.g. \cite{BL}.

Let $i: \Theta^{n} \to A^{n+1}$ be the natural embedding and $\mathcal D=c_1(\mathcal L) \in H^2(\Theta^{n},\mathbb Z)$ be the first Chern class of the line bundle $\mathcal L: =i^*(L),$ which is also the normal bundle of $\Theta^n \subset A^{n+1}$. Then the total Chern class $c(\Theta^n)=\sum_{k=0}^n c_k(\Theta^n)$  satisfies
\beq{rel}
c(\Theta^n)(1+\mathcal D)=1
\eeq
since the tangent bundle of an abelian variety is trivial. This means  that the topological Euler characteristic 
\beq{chi}
\chi(\Theta^n)=c_{n}(\Theta^n)=(-1)^n\mathcal D^n =(-1)^n (n+1)!,
\eeq
since $\mathcal D^n=(n+1)!$ (see next section). Alternatively, we can use formula (\ref{hirzeb}) with $Q(z)=1+z$ corresponding to the Euler characteristic $\Phi=\chi$:
$$
\sum_{n=0}^\infty\chi(\Theta^n)\frac{z^{n+1}}{(n+1)!} = \frac{z}{1+z}.
$$

The Betti numbers of the theta divisors $\Theta^{n}$ are not difficult to compute, see e.g. \cite{IW, NS}. 
Indeed, by the Lefschetz hyperplane theorem the embedding $i: \Theta^{n} \to A^{n+1}$ induces the isomorphisms
$$
i_*: H_k(\Theta^{n},\mathbb Z)\to H_k(A^{n+1},\mathbb Z), \quad
i_*:\pi_k(\Theta^{n}) \to \pi_k(A^{n+1})
$$
for $k<n$, while for $k=n$ these homomorphisms are surjections \cite{Lazarsfeld, Milnor}. 

In particular, for $n\geq 2$ the theta divisor has the fundamental group $$\pi_1(\Theta^n) = \pi_1(A^{n+1})=\mathbb Z^{2n+2}$$
and $\pi_k(\Theta^n)=\pi_k(A^{n+1})=0$ for $1<k<n$ since $A^{n+1}$ is a topological torus. When $n=1$ the theta divisor $\Theta^1\cong\mathcal C\subset J(\mathcal C)$ is a genus 2 curve with non-commutative $\pi_1(\mathcal C)$ and $i_*:\pi_1(\mathcal C) \to \mathbb Z^4$ being the abelianisation map and all other homotopy groups being trivial.

The manifolds with free abelian fundamental group were studied by Novikov \cite{Nov-66} in relation with the famous problem of topological invariance of rational Pontryagin classes.
The theta divisors give non-trivial examples of such manifolds with non-zero Pontryagin classes.

Using Poincare duality we obtain now all Betti numbers of $\Theta^n$ as
$$
b_k(\Theta^n)=b_k(A^{n+1})={2n+2 \choose k}=b_{2n-k}(\Theta^n), \,\, k<n,
$$
except the middle one $b_n$, which can be found using the formula (\ref{chi}) for the Euler characteristic:
\beq{catalan}
b_{n}(\Theta^n)=(n+1)!+\frac{n}{n+2} {2n+2 \choose n+1}=(n+1)!+n C_{n+1},
\eeq
where $C_n=\frac{1}{n+1} {2n \choose n}$ is the $n$-th Catalan number, see \cite{Stanley}.

Since the cohomology groups of $\Theta^n$ have no torsion \cite{IW}, this defines them uniquely, but it seems that the multiplication structure remains to be understood.
Note that we can compute the signature $\tau(\Theta^n)$ of the corresponding quadratic form on the middle cohomology for even $n$ using our general formula (\ref{hirzeb}) (see also \cite{B-2020}). Indeed, the signature corresponds to the Hirzebruch $L$-genus with
$Q(z)=\frac{z}{\tanh z}.$ Since
$$
\frac{z}{Q(z)}=\tanh z=\sum_{k=0}^\infty 2^{2k+2}(2^{2k+2}-1)B_{2k+2}\frac{z^{2k+1}}{(2k+2)!}
$$
from (\ref{hirzeb}) we have that the signature of the theta divisor $\Theta^n$ for even $n$ is
\beq{sign}
\tau(\Theta^n)=\frac{2^{n+2}(2^{n+2}-1)}{n+2}B_{n+2},
\eeq
where $B_n$ are the classical Bernoulli numbers:
$$B_0 = 1, \, B_1 = - \frac{1}{2}, \, B_2 = \frac{1}{6}, \, B_4 = - \frac{1}{30}, \, B_6 =  \frac{1}{42},\, B_8 = -  \frac{1}{30}, \, B_{10} =  \frac{5}{66}, ...$$
In particular, for $n=2$ we have $b_2(\Theta^2)=16$ and 
$
\tau(\Theta^2)=\frac{2^4(2^4-1)B_4}{4}=-2. 
$
Note that the integrality of the right hand side of (\ref{sign}) is not obvious. The appearance of both Bernoulli and Catalan numbers looks quite intriguing and invites further study here.

We should mention here very interesting work by Nakayashiki and Smirnov on the computation of cohomology groups of the complement of the (singular) theta divisor in hyperelliptic Jacobi variety \cite{NS,Nak}. 

For all $n$ the smooth theta divisor $\Theta^n$ is a projective variety of general type. Indeed, since the canonical class of abelian variety is zero, by the adjunction formula \cite{GH} the canonical bundle
$K_{\Theta^n}=i^*(L)=\mathcal L$, which is ample. 
In particular, $\mathcal L$ is known to have an $(n+1)$-dimensional space of sections generated by the partial derivatives  $\partial_\xi \theta(z,\tau)$ of the theta function.
 
 Consider the intersection 
\beq{inters2}
\Theta_{k}^{n-k}=\Theta^n\cap \Theta^n(a_1) \cap \dots \Theta^n(a_k)
\eeq
of $\Theta^n$ with $k$ general translates $\Theta^n(a_i), \, a_i \in A^{n+1}$ of the theta divisor $\Theta^n.$ 
 
  \begin{prop}
  For all $k<n$ and general $a_i \in A^{n+1}, \, i=1,\dots,k$ the variety $\Theta_{k}^{n-k}$ is smooth and irreducible of general type.
  \end{prop}
 
\begin{proof}
The abelian variety $A^{n+1}$ is a group variety and hence is homogeneous.
Each of the subvarieties $\Theta^n(a_i)$ is a translate of the smooth subvariety $\Theta^n$. Therefore for general  general $a_i \in A^{n+1}$ the intersection $\Theta_{k}^{n-k}$ is smooth by Kleiman's theorem \cite{ful}. 

For $k<n$ this variety is irreducible, since by the Lefschetz hyperplane theorem (see Th.3.1.17 in \cite{Lazarsfeld}) the Betti number $b_0(\Theta_{k}^{n-k})=1$ and thus $\Theta_{k}^{n-k}$ is connected.

The translate $\Theta^n(a_i)$ is the zero set of section of the shifted line bundle $L(a_i)$, so the adjunction formula \cite{GH} shows that the canonical bundle  $K_{\Theta_{k}^{n-k}}$ is given by
$$
K_{\Theta_{k}^{n-k}}=\bigotimes_{i=0}^k L(a_i)|_{\Theta_{k}^{n-k}}
$$
with $a_0:=0.$ Since each line bundle $L(a_i)$ is ample, the bundle $K_{\Theta_{k}^{n-k}}$ is ample too, so $\Theta_{k}^{n-k}$ is of general type.
\end{proof}

%

In particular, $\Theta_n^0$ consists of $(n+1)!$ points and $\Theta_{n-1}^1$ is a curve with Euler characteristic $\chi=-n(n+1)!.$

The cobordism class $[\Theta_k^{n-k}]$ does not depend on $a_1,\dots,a_k$ and the choice of abelian variety. 
Alternatively,  as a representative of this cobordism class we can choose the variety given by the system of equations 
\beq{equa}
 \theta(z,\tau)=0, \, \partial_{\xi_1} \theta(z,\tau)=0, \dots, \partial_{\xi_k} \theta(z,\tau)=0, \quad z \in A^{n+1}.
\eeq
Since the partial derivatives $\partial_{\xi} \theta(z,\tau)$ are the sections of the line bundle $\mathcal L$ on $\Theta^n$, this bundle is base-point free by smoothness of $\Theta^n.$ Therefore by Bertini theorem \cite{GH} for generic $\xi_1,\dots,\xi_k \in \mathbb C^{n+1}$ the equation (\ref{equa}) determines a smooth irreducible algebraic variety for all $k<n.$



It turns out that the theta divisors can be used to represent Todd class of complex vector bundles in quite general situation. 

\begin{Theorem}
The characteristic Todd class of complex vector bundle $\xi$ over CW-complex $X$ is given by the formula
\beq{TODD}
Td_U(\xi)=\sum_\lambda c_\lambda(-\xi)\frac{[\Theta^\lambda]}{(\lambda+1)!},
\eeq
where the sum is over all partitions $\lambda=(i_1,\dots, i_k)$ and $\Theta^\lambda=\Theta^{i_1}\times\dots\times \Theta^{i_k}$ as in (\ref{prodt}).
In particular, for any $U$-manifold $M$ we have
\beq{TODDM}
Td_U(TM)=\sum_\lambda c_\lambda(\nu M)\frac{[\Theta^\lambda]}{(\lambda+1)!}.
\eeq
\end{Theorem}

\begin{proof}
Denote the right hand side of formula (\ref{TODD}) as $\mathbb T(\xi)$
and check that this characteristic class satisfies both defining properties of the Todd class.

The first property $\mathbb T(\xi_1+\xi_2)=\mathbb T(\xi_1)\mathbb T(\xi_2)$    follows from the well-known formula for the Chern classes
$$
c_\lambda(\xi_1+\xi_2)=\sum_{\lambda=(\lambda_1, \lambda_2)}c_{\lambda_1}(\xi_1)c_{\lambda_2}(\xi_2).
$$
The evaluation $(\mathbb T(TM),\langle M\rangle) \in \Omega_U\otimes \mathbb Q$ defines the homomorphism of $\Omega_U\otimes \mathbb Q$ into itself.
We claim that it is the identity.

\begin{Lemma}
The Chern numbers (\ref{chn}) of the theta divisor $\Theta^n$ are
\beq{c1l}
c^\nu_\lambda(\Theta^n)=0
\eeq
for any partition $\lambda$ of $n$ different from the one-part partition $\lambda=(n)$, and
\beq{c2l}
c^\nu_{(n)}(\Theta^n)=(n+1)!.
\eeq
\end{Lemma}

\begin{proof}
Since the tangent bundle of abelian variety is trivial, the normal bundle $\nu\Theta^n$ is stably equivalent to the line bundle $\mathcal L=i^*(L)$, where $i: \Theta^n \to A^{n+1}$ is natural embedding and $L$ is the principal polarisation line bundle on $A^{n+1}.$ This immediately implies (\ref{c1l}). 

To prove condition (\ref{c2l}) we need only to use the well-known fact that $$D^{g}=g!\in H^{2g}(A^g, \mathbb Z)=\mathbb Z$$ where $D \in H^2(A^g,\mathbb Z)$ is the Poincare dual cohomology class of the theta divisor $\Theta \subset A^g$ of any principally polarised abelian variety (see e.g. \cite{BL}). Geometrically, this means that the intersection of $g$ generic shifts of theta divisor $\Theta$ of abelian variety $A^{g}$ consists of $g!$ points. One can see this easily in the degenerate case when $X^g=\mathcal E^g$ is the product of $g$ elliptic curves.
\end{proof}

Due to the results of Milnor and Novikov, since $c_{(n)}(\Theta^n)\neq 0$ the theta divisors can be chosen as multiplicative generators of the algebra $\Omega_U\otimes \mathbb Q.$
Hence to prove that $(\mathbb T(TM),\langle M\rangle)=[M]$ it is enough to check this for all theta divisors, which immediately follows from the lemma.
By uniqueness $\mathbb T(\xi)=Td_U(\xi),$ which completes the proof.
\end{proof}

\begin{Corollary}
The cobordism class of any $U$-manifold $M^{2n}$ can be given by formula (\ref{decom}).
\end{Corollary}

Indeed, we have formula (\ref{decom}) by evaluating formula (\ref{TODDM}) on the fundamental cycle of $M^{2n}$ and using the second property of the Todd class.

Recall now that the Todd class is uniquely defined by its value on the line bundle $\eta$ over $\mathbb{C}P^\infty$ by formula (\ref{tda}):
$$
Td_U(\eta)=1+\sum_{n \in \mathbb N}A_nz^n, \quad z=c_1(\eta) \in H^2(\mathbb{C}P^\infty,\mathbb Z)
$$
where some coefficients $A_n\in \Omega^{-2n}_U\otimes \mathbb Q$.
Now we can describe these coefficients in terms of theta divisors.


\begin{Corollary}
The Todd class of $\eta$ has the form
\beq{td}
Td_U(\eta)=\left( \sum_{n=0}^\infty\frac{[\Theta^n]}{(n+1)!}z^n\right)^{-1}.
\eeq
\end{Corollary}


Indeed, this follows from the relation $Td_U(\eta)Td_U(-\eta)=1$ and  formula (\ref{TODD}) applied to $\xi=-\eta.$

Now we are ready to prove Theorem 1.1. 

\begin{Theorem}
The Chern-Dold character is uniquely defined by the formula
\beq{CDN}
ch_U(u)=z+\sum_{n=1}^\infty[\Theta^n]\frac{z^{n+1}}{(n+1)!},
\eeq
where $z=c_1(\eta)\in H^2(\mathbb CP^N, \mathbb Z)$ and $u \in U^2(\mathbb CP^N)$ is the first Chern class of the line bundle $\eta$
over $\mathbb CP^N$ in the complex cobordisms.
\end{Theorem}

\begin{proof}
We use the fundamental relation (\ref{tdchcob}), which in these notations has the form
\beq{fun2}
Td_U(\eta)ch_U(u)=z.
\eeq
Comparing this formula with (\ref{td}) we have formula (\ref{CDN}) and the claim. 
\end{proof}

 As a corollary we have formula (\ref{hirzeb}) for the exponential generating function of any Hirzebruch genus $\Phi$ of all theta divisors.


The Hirzebruch genus in complex cobordisms \cite{Hirz-66} is a homomorphism $\Phi: \Omega_U \to \mathcal A$, where $\mathcal A$ is some algebra over $\mathbb Q$, determined by its characteristic power series
\beq{Q}
Q(z)=1+\sum_{n=1}^\infty a_n z^n, \quad a_n \in \mathcal A.
\eeq
For any Hirzebruch genus $\Phi$ with characteristic power series $Q(z)$ we have the following relations with the Chern-Dold character and Todd class
\beq{CDNh}
\Phi(ch_U(u))=\frac{z}{Q(z)}, \quad \Phi(Td_U(\eta))=Q(z)
\eeq
(see \cite{B-1970}). Combining this with Theorem 3.6 we have

\begin{Corollary}
The exponential generating function of $\Phi(\Theta^n)$ is 
\beq{hirzn}
z+\sum_{n=1}^\infty \Phi(\Theta^n)\frac{z^{n+1}}{(n+1)!}=\frac{z}{Q(z)}.
\eeq
\end{Corollary}

Consider now the special case of $\mathcal A=\Omega \otimes \mathbb Q$ and $\Phi: \Omega \to \mathcal A=\Omega \otimes \mathbb Q$ being the natural embedding.
Then from (\ref{CDNh}) we see that the corresponding series $Q(z)=z/\beta(z)$ with $\beta(z)$ given by (\ref{CDnew}).

Define the cobordism classes $v_n \in \Omega_U\otimes \mathbb Q$ as the coefficients of $Q_{v}(z)$ written in the form
\beq{qv}
Q_{v}(z)=1+\sum_{n=1}^\infty (-1)^nv_n \frac{z^n}{(n+1)!},
\eeq
with the series $Q_v$ defined by 
\beq{relQ}
\left(1+\sum_{n=1}^\infty (-1)^nv_n \frac{z^n}{(n+1)!}\right)\left(1+\sum_{n=1}^\infty [\Theta^n] \frac{z^n}{(n+1)!}\right)\equiv 1.
\eeq
In the theory of symmetric functions \cite{Mac} this corresponds to the duality $\omega$ between elementary symmetric functions $e_n$ and complete symmetric functions $h_n$ 
(see formula (2.6) in \cite{Mac}), where we substitute 
\beq{enhn}
e_n=\frac{\Theta^n}{(n+1)!}, \quad h_n=\frac{v_n}{(n+1)!}.
\eeq

The determinantal formula for this duality (see e.g. page 28 in \cite{Mac})
$$
h_n=\det (e_{1-i+j})_{1\leq i,j \leq n}
$$
allows to express $v_n$ as a polynomial of $t_k:=[\Theta^k], \,k=1,\dots, n$ with rational coefficients.
For example, 
 \beq{ex1}
 v_1=t_1, \,\, v_2=-t_2+\frac{3}{2}t_1^2, \,\, v_3=t_3-4t_1t_2+3t_1^3,
 \eeq
 \beq{ex2}
 v_4=-t_4+5t_1t_3-15t_1^2t_2+\frac{10}{3}t_2^2+\frac{15}{2}t_1^4,
\eeq
\beq{ex3}
 v_5=t_5-6t_1 t_4+30t_1 t_2^2-60t_1^3t_2-10t_2t_3+\frac{45}{2}t_1^2t_3+\frac{45}{2}t_1^5.
\eeq
In fact, since the series in (\ref{relQ}) are the exponential generating functions of $y_n=\frac{v_n}{n+1}$ and $x_n=\frac{[\Theta^n]}{n+1}$ respectively, we can apply the results about Hurwitz series \cite{Hur,PS} to deduce that
$y_n\in \mathbb Z[x_1,\dots,x_n]$ is a polynomial with integer coefficients (which is clearly not the case in the formulae for $v_n$ above).
 
 Let $v_\lambda=v_{i_1}\dots v_{i_k}, \, \lambda=(i_1,\dots,i_k)$ then the cobordism classes $v_\lambda$ considered as elements of $S^*$ form a basis dual to the Landweber-Novikov operations $\bar S_\lambda$ defined using the tangent bundles (see Novikov \cite{Nov-67}). In particular, the (usual, tangent) characteristic numbers $c_\lambda(v_n)=0$ for any $\lambda\neq (n)$, and $$c_{(n)}(v_n)=(-1)^n(n+1)!.$$
 
\begin{Theorem}
For any $U$-manifold $M^{2n}$ we have the following analogue of formula (\ref{decom}):
\beq{decomt}
[M^{2n}]=\sum_{\lambda: |\lambda|=n}(-1)^{|\lambda|}c_\lambda(M^{2n})\frac{v_\lambda}{(\lambda+1)!},
\eeq
where instead of $c^\nu_\lambda(M^{2n})$ we use the characteristic numbers $c_\lambda(M^{2n})$ of the tangent bundle $TM^{2n}.$

The Hirzebruch genus $\Phi(v_n)$ with characteristic power series $Q(z)$ can be found from the generating function 
\beq{hirznn}
1+\sum_{n=1}^\infty (-1)^n\Phi(v_n)\frac{z^{n}}{(n+1)!}=Q(z).
\eeq
\end{Theorem}
 
 The proof follows directly from Corollary 3.7 and formulas (\ref{qv}),(\ref{relQ}).
 
 In particular, for the Euler characteristic with $Q(z)=1+z$ we conclude that $\chi(v_n)=0$ for all $n>1$ with $\chi(v_1)=-2.$
Similarly, for the Todd genus we have
 $$
 1+\sum_{n=1}^\infty (-1)^nTd(v_n) \frac{z^n}{(n+1)!}=\frac{z}{1-e^{-z}}=\sum_{n=0}^\infty (-1)^n B_n\frac{z^n}{n!},
 $$
where $B_n$ are the Bernoulli numbers. This implies that
\beq{todtn}
Td(v_n)=(n+1)B_n
\eeq
is zero for odd $n>1$ and non-zero rational for even $n$. In particular, $Td(v_2)=\frac{1}{2}$, so the cobordism class $v_2$ cannot be represented by a $U$-manifold.

A natural question is what is the minimal integer $k_n\in \mathbb N$ such that $
k_n v_n \in \Omega_U$
is a cobordism class of some $U$-manifold.
The following result gives the answer to this question.

Let $q_n$ be the denominator of the fraction $(n+1)B_n$ written in the simplest form, where for odd $n>1$ with $B_n=0$ we put $q_n=1.$

\begin{Theorem}
The cobordism class 
\beq{vn}
q_n v_n =[V^{2n}]\in \Omega_U
\eeq for some $U$-manifold $V^{2n}$, and this is the smallest multiple of $v_n$ with such property.
\end{Theorem}

\begin{proof}
Let $k_n v_n =[M^{2n}]\in \Omega_U$, then the Todd genus $k_n(n+1)B_n$ must be an integer, which means that $k_n$ is divisible by $q_n,$ and thus $k_n=q_n$ must be minimal.

To prove that the cobordism class $
q_n v_n =[V^{2n}]$ for some $U$-manifold $V^{2n}$ we use the fundamental result of Hattori \cite{Hat-66} and Stong \cite{Stong-65, Stong-68}, saying that element $a \in \Omega_U\otimes \mathbb Q$ 
belongs to $\Omega_U$ if and only if the Todd genus of all the results $S_\lambda(a)$ of the Landweber-Novikov operations applied to $a$ are integer.
This is obviously true for the Todd genus of $a=q_n v_n$ since $Td(a)=q_n(n+1)B_n.$  We have the following important result. 

\begin{Lemma}
Every Landweber-Novikov operation $S_\lambda$, different from the identity, maps the cobordism class $v_n$ to the subring $\Theta_U=\mathbb Z[t_1,\,t_2, \dots ]\subset\Omega_U$ generated by the theta divisors $t_k=[\Theta^k], \, k \in \mathbb N.$

In particular, $S_{(1)} (v_1)=-2$, $S_{(1)} (v_n)=0$ for $n>1$ and for $n\geq 2$ 
$
S_{(2)} (v_n)=-n(n+1)[\Theta^{n-2}].
$ 
\end{Lemma}

We will use the following properties of the Landweber-Novikov operations: 
\beq{bet}
S_{(k)} (\beta(z))=\beta(z)^{k+1},
\eeq
where $\beta(z)$ is given by (\ref{CDnew}), and $S_{\lambda} (\beta(z))=0$ for all non one-part partitions (see the next section for more details).

We will prove Lemma by induction in the length of $\lambda.$ 
For length one partition $\lambda=(k)$ we apply the operation $S_{(k)}$ to the relation (\ref{relQ}), which we rewrite as $Q_v(z) \beta(z)=z$, to have
$$
S_{(k)} (Q_v(z) \beta(z))=S_{(k)}Q_v(z) \beta(z)+Q_v(z) \beta(z)^{k+1}=0,
$$
where we have used (\ref{bet}). Since $Q_v(z) \beta(z)=z$ this implies that
$$
S_{(k)} (Q_v(z))=\sum_{n=1}^\infty (-1)^n S_{(k)} (v_n) \frac{z^n}{(n+1)!}=-z\beta(z)^{k-1}.
$$
In particular, for $k=1$ we have $S_{(1)} (Q_v(z))=-z$, so $S_{(1)} (v_n)=0$ for $n>1$ and $S_{(1)} (v_1)=-2.$
For $k=2$ we have $S_{(2)} (Q_v(z))=-z\beta(z)$, which implies that for $n\geq 2$
$
S_{(2)} (v_n)=-n(n+1)[\Theta^{n-2}].
$
When $k>2$ we can write
$$
S_{(k)} (Q_v(z))=-z\beta(z)^{k-1}=-z
S_{(k-2)} (\beta(z))
$$
and use the induction in $k$ to conclude the proof for one-part partitions. 

The general case follows from the multiplicative property (\ref{multS}) of the Landweber-Novikov operations and the fact that $S_{\lambda} (\beta(z))=0$ for all partitions $\lambda$ of length more than one.

Now the claim of the Theorem follows from Lemma and Stong and Hattori results, since the Todd genus of any $U$-manifold is integer.
\end{proof}

\begin{Remark} It is natural to introduce also the following realisation of the power sums in complex cobordisms $p_n=\frac{(-1)^{n-1}}{(n-1)!}w_n$, where $w_n\in \Omega_U\otimes \mathbb Q$ are defined by the formula
\beq{wn}
\beta(z)=z \exp\left(\sum_{n=1}^\infty w_n\frac{z^n}{n!}\right)
\eeq
(see \cite{Mac} and formulae  (\ref{qv}), (\ref{enhn}) above).
Applying the Landweber-Novikov operation $S_{(k)}$ to both sides and using (\ref{LN1}) we have
$$
\beta(z)^k=\sum_{n=1}^\infty [\Theta_{k-1}^{n-k}] \frac{z^n}{n!}=\sum_{n=1}^\infty S_{(k)}(w_n) \frac{z^n}{n!},
$$
which implies that
\beq{wns1}
S_{(k)}(w_n)=[\Theta_{k-1}^{n-k}]=S_{(k-1)}(\Theta^{n-1}).
\eeq
Up to a multiple, $w_n$ coincide with the cobordism classes of the manifolds $N^{2n}$ from Theorem 1.4 of the paper \cite{B-1970}.
\end{Remark}

\section{The Landweber-Novikov algebra and quantisation of complex cobordisms }

The Landweber-Novikov algebra introduced in \cite{L-1967, Nov-67} is an important subalgebra of all cohomological operations in complex cobordisms with additive basis given by the Landweber-Novikov operations $S_\lambda, \lambda \in \mathcal P$, where  $\mathcal P$ is the set of all partitions $\lambda=(i_1,\dots, i_k)$ (see \cite{Mac}).

Recall that the cobordism class $\alpha \in U^2(X)$ is called {\it geometric} if it belongs to the image of the natural homomorphism
$H^2(X,\mathbb Z) \to U^2(X)$ (see \cite{Nov-67}).
The action of the Landweber-Novikov operations on any geometric cobordism class $\alpha$ is defined as follows:
for any one-part partition $\lambda=(k), \,k\in \mathbb N$ the Landweber-Novikov operation $S_{(k)}(\alpha)=\alpha^{k+1},$ 
while for all other partitions $S_\lambda(\alpha)=0$  (see Lemma 5.6 in \cite{Nov-67}). 
Together with the rule
\beq{multS}
S_\lambda(xy) = \sum_{\lambda=(\lambda_1,\lambda_2)}S_{\lambda_1}(x)\otimes S_{\lambda_2}(y), \quad x,y \in U^*(X)
\eeq
this uniquely defines the operations $S_\lambda: U^*(X) \to U^*(X)$ for any $CW$-complex $X.$

The action of $S_\lambda$ on the cobordism class of $U$-manifold $M^{2n}$ can be given as
\beq{slm}
S_\lambda([M^{2n}])=p_\sharp(c_\lambda^U(\nu M^{2n})),
\eeq
where $p_\sharp: U^*(M^{2n})\to \Omega_U$ is the Gysin homomorphism for $p: M^{2n}\to pt$ and $c_\lambda^U(\nu M^{2n})\in U^{2|\lambda|}(M^{2n})$ are the Conner-Floyd characteristic classes of the normal bundle of $M^{2n}$ \cite{CF}. 

In particular, when $|\lambda|=n$  we have $c_\lambda^U(\nu M^{2n})\in U^{2n}(M^{2n})=\mathbb Z$ and  
\beq{slmn}
S_\lambda([M^{2n}])=c_\lambda^\nu(M^{2n})
\eeq 
are the characteristic numbers (\ref{chn}).

In this paper we consider the Landweber-Novikov algebra as the graded algebra $S$ over $\mathbb Q$ with a special basis $S_\lambda, \lambda \in \mathcal P,$ where the degree of $S_\lambda$ equals to $2|\lambda|=2(i_1+\dots+i_k)$ and the identity is defined as $S_\emptyset.$
It is also a Hopf algebra with the diagonal given by 
\beq{diag}
\triangle S_\lambda = \sum_{\lambda=(\lambda_1,\lambda_2)}S_{\lambda_1}\otimes S_{\lambda_2}.
\eeq
The diagonal is symmetric, so the dual graded Hopf algebra $S^* = Hom(S,\mathbb Q)$ is commutative. 

The following result, which can be extracted from \cite{L-1967,Nov-67}, plays an important role in constructions from \cite{BSh}.

\begin{prop} 
There is a canonical isomorphism of algebras 
\beq{sig}
\sigma: S^*\cong \Omega_U\otimes \mathbb Q.
\eeq
\end{prop}
\begin{proof}
Let $\mu: U^*(X) \to H^*(X,\mathbb Z)$ be the {\it cycle realisation homomorphism}, which is defined uniquely by the property 
$\mu(u)=z,$ where as before $u \in U^2(\mathbb CP^\infty)$ and $z\in H^2(\mathbb CP^\infty, \mathbb Z)$ be the first Chern classes of the universal line bundle $\eta$
over $\mathbb CP^\infty$ in the complex cobordisms and cohomology theory respectively. In particular, when $X=pt$ is a point we have the {\it augmentation ring homomorphism}
$\mu: \Omega_U \to \mathbb Z$, which is 1 on the identity and 0 on elements with non-zero grading.

We have the canonical graded pairing between Landweber-Novikov algebra $S$ and $\Omega_U\otimes \mathbb Q$:
\beq{pair}
\phi: S\otimes (\Omega_U\otimes \mathbb Q) \to \mathbb Q, \quad S_\lambda \otimes [M] \to \mu(S_\lambda ([M])).
\eeq
which is known to be non-degenerate \cite{Nov-67}. This implies the canonical isomorphism  
$
\sigma: S^* \cong \Omega_U\otimes\mathbb Q,
$
which is also the algebra isomorphism as it follows from (\ref{multS}) and (\ref{diag}).
\end{proof}
 
 A natural question is what is the basis $S^\lambda \in \Omega_U\otimes\mathbb Q$ dual to the Landweber-Novikov operations $S_\lambda,\, \lambda \in \mathcal P.$ We can now give an explicit answer to this question.
 
\begin{Theorem}
The dual basis to Landweber-Novikov operations $S_\lambda, \, \lambda \in \mathcal P$ in $\Omega_U\otimes\mathbb Q$ is given by
\beq{Sstar}
\sigma(S^\lambda)=\frac{[\Theta^\lambda]}{(\lambda+1)!},
\eeq
where $\Theta^\lambda$ is the product of theta divisors (\ref{prodt}).
\end{Theorem}

\begin{proof}
Indeed, from formula (\ref{slmn}) and Lemma 3.2 we have 
$S_\lambda([\Theta^n])=0$
for any partition $\lambda$ of $n$ different from $\lambda=(n)$, and
$S_{(n)}([\Theta^n])=(n+1)!.$
Now the claim follows from the multiplicativity property of the Landweber-Novikov operations.
\end{proof}

We can now describe the action of the Landweber-Novikov operations on the theta divisors and to prove Theorem 1.2.

Let $[\Theta_{k}^{n-k}]$ be the cobordism class of the intersection divisor (\ref{inters}).
\begin{Theorem}
For the one-part partitions 
$\lambda=(k), \, k\leq n$ we have
\beq{LN1}
S_{(k)}([\Theta^n])=[\Theta_{k}^{n-k}],
\eeq
while for all other partitions $S_\lambda[\Theta^n]=0.$
\end{Theorem}

\begin{proof}
Recall that the normal bundle $\nu \Theta^n$ can be identified with $\mathcal L=i^*(L)$, where $i: \Theta^n \to A^{n+1}$ and $L$ is the principal polarisation bundle of $A^{n+1}$. 
Since $\mathcal L$ is of rank one from formula (\ref{slm}) and properties of Conner-Floyd characteristic classes it immediately follows that $S_\lambda([\Theta^n])=0$ if $\lambda$ is not a one-part partition. 

For the proof of formula (\ref{LN1}) we need the following properties of the Gysin homomorphism $f_\sharp: U^*(M_1)\to U^*(M_2)$ for any mapping of $U$-manifolds $f: M_1 \to M_2$
(see e.g. Proposition D.3.6 in \cite{BT}):
\beq{p2} 
f_\sharp(x f^*(y))=f_\sharp(x)y, \quad x\in U^*(M_1), \, y \in U^*(M_2)
\eeq
and for $f: M_1 \to M_2$ and $g: M_2 \to M_3$ we have
$(gf)_\sharp=g_\sharp f_\sharp.$

Let $[1]\in U^0(\Theta^n)$ be the cobordism class Poincar\'e-Atiyah dual to the identity map of $\Theta^n$ to itself, $[D] \in U^2(A^{n+1})$ be the cobordism class dual to the bordism class $\{\Theta^n\}\in U_{2n}(A^{n+1})$ defined by the embedding $i:\Theta^n \to A^{n+1},$ 
then by construction $i_\sharp([1])=[D].$ Note that $[D]=c_1^U(L)$ is the first Conner-Floyd class of the line bundle $L.$

Now for the one-part partition $\lambda=(k), \, k\leq n$ we have $$c_{(k)}^U(\nu \Theta^n)=c_{(k)}^U(\mathcal L)=c_{(k)}^U(i^*(L))=i^*(c_1^U(L)^k)=i^*[D]^k.$$  By (\ref{slm}) and the properties of Gysin homomorphisms  $$S_{(k)}([\Theta^n])=p_\sharp(i^*[D]^k)=f_\sharp i_\sharp(i^*[D]^k),$$ where $f$ and $p=f\circ i$ are the mappings to a point of $A^{n+1}$ and $\Theta^n$ respectively. Since $[1]$ is the identity in the ring $U^*(\Theta^n)$ we have using property (\ref{p2}) with $x=[1]$ and $y=[D]^k$ that $$i_\sharp(i^*[D]^k)=i_\sharp([1]\cdot i^*[D]^k)=i_\sharp([1]) [D]^k=[D]^{k+1},$$
and thus $S_{(k)}([\Theta^n])=f_\sharp([D]^{k+1}).$ By the definition of the Gysin homomorphism the cobordism class $f_\sharp([D]^{k+1})$ is dual to the transversal intersection of $k+1$ copies of cycles dual to $[D]$, which can be realised as the intersection of generally shifted theta-divisors (\ref{inters}). This means that
$
f_\sharp([D]^{k+1})=[\Theta_{k}^{n-k}],
$
which proves the theorem.
\end{proof}

Let us introduce now the {\it quantum complex cobordism theory} as the extraordinary cohomology theory $ QU^*:=U^*\otimes S^*$ with $QU^*(pt)=Q\Omega^*:=\Omega_U\otimes   S^*$.
The algebra $S^*$ is used here as the deformation parameter space.

Consider the following {\it quantisation map} $q^\star: U^*(X)\to U^*(X)\otimes S^*$
inspired by \cite{B-1995}:
\beq{qCDf}
q^\star(x)=x\otimes 1+\sum_{\lambda}S_\lambda(x)\otimes S^\lambda.
\eeq
In particular, if $X=pt$ then $U^*(X)=\Omega_U$ and for any $U$-manifold $M^{2n}$ we have
\beq{qCDM2n}
q^\star(M^{2n})=[M^{2n}]\otimes 1+\sum_{\lambda: |\lambda|<n}S_\lambda([M^{2n}])\otimes S^\lambda + 1\otimes\sum_{\lambda: |\lambda|=n}c^\nu_\lambda(M^{2n}) S^\lambda.
\eeq
Here we have used that when $|\lambda|=n$ we have $S_\lambda ([M^{2n}])=c^\nu_\lambda(M^{2n})$ (see Novikov \cite{Nov-67}).

Introduce also the following quantum analogue of cycle realisation homomorphism as
\beq{mu*}
\mu^\star: U^*(X)\otimes S^* \to H^*(X, \Omega_U\otimes\mathbb Q), \quad \mu^\star=\mu \otimes \sigma,
\eeq
where we used the natural isomorphism $$H^*(X,\mathbb Z)\otimes (\Omega_U\otimes \mathbb Q)\cong H^*(X,\Omega_U\otimes \mathbb Q).$$
It can be viewed as a kind of ``dequantisation" map.

We claim that the Chern-Dold character in complex cobordisms is the composition of the quantisation and dequantisation maps:
$$
U^*(X)\stackrel{q^\star}{\longrightarrow}U^*(X)\otimes S^* \stackrel{\mu^\star}{\longrightarrow} H^*(X, \Omega_U\otimes\mathbb Q).
$$



\begin{Theorem}
The quantisation map defined by (\ref{qCDf}) is the algebra homomorphism.
The Chern-Dold character $ch_U: U^*(X)\to H^*(X, \Omega_U\otimes\mathbb Q)$ is the composition 
\beq{chmu}
ch_U=\mu^\star \circ q^\star.
\eeq
\end{Theorem}

\begin{proof}
The algebra homomorphism property $q^\star(xy)=q^\star (x) q^\star (y)$ follows from the multiplicative property  (\ref{multS}) of the Landweber-Novikov operations.

To prove the rest it is enough to check (\ref{chmu}) only for $u=c_1^U(\eta) \in U^2(\mathbb CP^\infty)$. In that case by definition
$$q^\star(u)=u\otimes 1+\sum_{\lambda}S_\lambda(u)\otimes S^\lambda=u\otimes 1+\sum_{n \in \mathbb N}S_{(n)}(u)\otimes S^{(n)},$$
since $u$ is geometric. But $S_{(n)}(u)=u^{n+1}$, so
$q^\star(u)=u\otimes 1+\sum_{n \in \mathbb N}u^{n+1} \otimes S^{(n)}.$
Now we use that $\mu(u)=z=c_1(\eta)$ and Theorem 4.2 saying that $\sigma(S^{(n)})=\frac{[\Theta^n]}{(n+1)!}$
to conclude that
$$
\mu^\star \circ q^\star(u)=z+\sum_{n \in \mathbb N}z^{n+1} \frac{[\Theta^n]}{(n+1)!}=ch_U(u)
$$
due to Theorem 3.6.
\end{proof}

\begin{Corollary}
The Chern-Dold character in complex cobordisms can be written as
\beq{CDmu}
ch_U(x)=\mu(x)+\sum_{\lambda}\mu(S_\lambda x) \frac{[\Theta^\lambda]}{(\lambda+1)!}, \quad x\in U^*(X)
\eeq
where $\mu: U^*(X) \to H^*(X,\mathbb Z)$ is the cycle realisation homomorphism.

For $X=pt$ and $x=[M^{2n}] \in \Omega_U$ we have 
\beq{6}
ch_U([M^{2n}])=\sum_{\lambda: |\lambda|=n}c^{\nu}_\lambda(M^{2n})\frac{[\Theta^\lambda]}{(\lambda+1)!}=[M^{2n}]
\eeq
in agreement with (\ref{CDpt}) and (\ref{decom}). 
\end{Corollary}

Note that the composition of the quantisation and dequantisation maps for $X=pt$ is non surprisingly the identical map, but leads to a non-trivial formula (\ref{6}).

According to general construction of Dold \cite{Dold} every extraordinary cohomology theory $h^*$ has its analogue of Chern character, which is the transformation of cohomology theories
$$
ch_h: h^*(X) \to H^*(X, \Omega_h\otimes \mathbb Q), \quad \Omega_h=h^*(pt)
$$
with characteristic property that for $X=pt$ it is the canonical homomorphism $\Omega_h \to \Omega_h\otimes \mathbb Q.$

In our case of quantum complex cobordism theory $QU^*$ the corresponding {\it quantum Chern-Dold character} $ch_U^\star: QU^*(X)=U^*\otimes S^*\to H^*(X, Q\Omega^*)$ has the following explicit form
\beq{qcd}
ch_U^\star(x\otimes s)=ch_U(x)\otimes s + \sum_{\lambda}S_\lambda(ch_U(x))\otimes S^\lambda s, \quad x \in U^*(X), \, s \in S^*.
\eeq


Let us deduce now our formula (\ref{res}). 

\begin{Theorem}
The cobordism class of $\Theta_{k}^{n-k}$ can be given as the residue integral at zero
$$
[\Theta_{k}^{n-k}]=\frac{(n+1)!}{2\pi i} \oint \beta(z)^{k+1} \frac{dz}{z^{n+2}},
$$
where $\beta(z)$ is given by (\ref{CDnew}). 
\end{Theorem}

\begin{proof}
We use the fact that the Chern-Dold character $ch_U: U^*(X)\to H^*(X, \Omega_U \otimes \mathbb Q)$ commutes with the action of the Landweber-Novikov algebra, which is trivial on the cohomology $H^*(X, \mathbb Q)$ (cf. \cite{B-1970}).

Applying Landweber-Novikov operation $S_\lambda$ to both sides of the relation (\ref{CDN}) we have
$S_\lambda ch_U(u)=ch_U(S_\lambda u)=\beta(z)^{k+1}$ if $\lambda= (k)$ for some $k\in \mathbb  N$ and 0 otherwise.
On the right hand side for $\lambda= (k)$ we have $$\sum_{n=k}^\infty S_{(k)}[\Theta^n]\frac{z^{n+1}}{(n+1)!}=\sum_{n=k}^\infty [\Theta_{k}^{n-k}]\frac{z^{n+1}}{(n+1)!},$$
and zero for any $\lambda\neq (k)$.
Comparison proves the claim and formula (\ref{res}).
\end{proof}

\begin{Corollary}
The cobordism class $[\Theta_{k}^{n-k}]$ is a polynomial of $[\Theta^1], \dots, [\Theta^{n-k}]$ with positive integer coefficients.

The polynomial subring $\Theta_U\subset \Omega_U$ generated by the theta divisors:
$$
\Theta_U=\mathbb Z[t_1,\dots, t_n,\dots], \, t_k=[\Theta^k], \, k \in \mathbb N,
$$ 
is invariant under the Landweber-Novikov operations.
\end{Corollary}

Indeed, since $\beta(z)$ is an exponential generating function, this follows from the properties of Hurwitz series \cite{Hur}, see also problems 174-177 in \cite{PS}, Vol. 2, Ch. 8.
In particular,
$$
[\Theta_1^{n-1}]=\sum_{k=0}^{n-1} {n+1 \choose k+1} [\Theta^k][\Theta^{n-k-1}], \quad [\Theta_{n-1}^1]=\frac{n(n+1)}{2}n! \,[\Theta^1].
$$

The following important interpretation of the Landweber-Novikov algebra was found by Buchstaber and Shokurov \cite{BSh}.

Consider the group $Diff_1$ of the formal diffeomorphisms of the line given by $$f(x)=x+\sum_{k\in \mathbb N} \alpha_k x^{k+1}, \quad x, \alpha_k \in \mathbb R.$$
Its Lie algebra $\mathfrak{diff}_1$ is the Lie algebra of the corresponding formal vector fields.

\begin{Theorem} (Buchstaber and Shokurov \cite{BSh})
The real version $S\otimes \mathbb R$ of the Landweber-Novikov algebra is isomorphic to the universal enveloping algebra of the Lie algebra $\mathfrak{diff}_1$, which can be identified with the algebra of the left-invariant differential operators on the group $Diff_1$.
\end{Theorem}

In particular, as an algebra $S\otimes \mathbb R$ is generated by two Landweber-Novikov operations $S_{(1)}$ and $S_{(2)}$, corresponding to two vector fields on the group $Diff_1$, which can be written in coordinates $\alpha_k, \, k \in \mathbb N$  as
\beq{vects}
S_{(1)}=\frac{\partial}{\partial \alpha_1}+\sum_{k=2}^\infty k\alpha_{k-1}\frac{\partial}{\partial \alpha_k}, \,\, S_{(2)}=\frac{\partial}{\partial \alpha_2}+\sum_{k=3}^\infty (k-1)\alpha_{k-2}\frac{\partial}{\partial \alpha_k}.
\eeq

The algebra $A^U$ of all cohomological operations in complex cobordisms was introduced in \cite{Nov-67} and is known as {\it Novikov algebra.} It is $\mathbb Z$-graded algebra, which is a free left $\Omega_U$-module with generators $S_\lambda$ and the commutation relations
\beq{comm}
S_\lambda \cdot x=\sum_{\lambda=(\lambda_1,\lambda_2)}S_{\lambda_1}(x) S_{\lambda_2}, \quad x \in \Omega_U
\eeq
(see Lemma 5.4 in Novikov \cite{Nov-67}). 

Let us introduce two new families of elements of Novikov algebra
\beq{bases}
T_{\lambda, \mu}=[\Theta^\lambda] S_\mu, \quad V_{\lambda,\mu}=[V^\lambda] S_\mu, \,\, \lambda,\mu \in \mathcal P,
\eeq
where for partition $\lambda=(\lambda_1,\dots, \lambda_k)$ we define $V^\lambda:=V^{2\lambda_1}\times\dots\times V^{2\lambda_k}$ with $V^{2n}$ being $U$-manifolds (\ref{vn}) from the previous section. 

From our results it follows that they give additive bases in two proper subalgebras of Novikov algebra, generating this algebra over $\mathbb Q$. Our Theorem 1.2, Corollary 4.7 and Lemma 3.9 can be used to describe the multiplication in these bases, which could be useful to study the representations of Novikov algebra.

The important elements of Novikov algebra, motivated by the Conner and Floyd results \cite{CF}, describing the $K$-theory in terms complex cobordisms, are the {\it Adams-Novikov operations} $\Psi^k_U$ (see \cite{Nov-67}). They are defined uniquely as the multiplicative operations in complex cobordisms, which satisfy the relation
\beq{psik}
\Psi^k_U(u)=\frac{1}{k} \beta(k \beta^{-1}(u)), \quad u=c_1^U(\eta) \in U^2(\mathbb CP^\infty).
\eeq
For any cobordism class $[M^{2n}]$ of $U$-manifold $M^{2n}$ we have
\beq{psik2}
\Psi^k_U([M^{2n}])=k^{n}[M^{2n}].
\eeq

Let $\Theta^n(k)$ be a smooth zero locus of generic section of the $k$-th tensor power $L^{\otimes k}$ of the line bundle $L,$ defining the principal polarisation of an abelian variety $A^{n+1}$.
Note that for $k\geq 2$ the zero locus of generic section of $L^{\otimes k}$ is smooth for any abelian variety by Bertini theorem. We would like to mention that the subvarieties $\Theta^n(k) \subset A^{n+1}$ have appeared in \cite{Mir} as the spectral varieties of certain commutative algebras of differential operators.

The cohomology and the corresponding cobordism class $[\Theta^n(k)]$ do not depend on the choice of such section.
As for the theta divisor, by Lefschetz hyperplane theorem the corresponding Betti numbers are
$$
b_j(\Theta^n(k))=b_j(A^{n+1})={2n+2 \choose j}=b_{2n-j}(\Theta^n(k)), \,\, j<n,
$$
except the middle one $b_n$, which can be found using the formula for the Euler characteristic $\chi(\Theta^n(k))=(-1)^nk^{n+1}(n+1)!:$
\beq{catalank}
b_{n}(\Theta^n(k))=k^{n+1}(n+1)!+\frac{n}{n+2} {2n+2 \choose n+1}.
\eeq
The signature of $\tau(\Theta^n(k)=k^{n+1}\tau(\Theta^n)$, so for even $n$ we have
\beq{sign2}
\tau(\Theta^n(k))=\frac{2^{n+2}(2^{n+2}-1)k^{n+1}}{n+2}B_{n+2},
\eeq
where $B_n$ are the Bernoulli numbers.

\begin{Theorem} The cobordism class $[\Theta^n(k)]$ can be expressed as
\beq{cobt}
[\Theta^n(k)]=k^{n+1}[\Theta^{n}]=k\Psi^k_U([\Theta^{n}]).
\eeq
\end{Theorem}

\begin{proof}
The normal bundle to $\Theta^n(k)$ can be identified with $\mathcal L^{\otimes k}=i^*L^{\otimes k}$ induced from $L^{\otimes k}$ by the embedding $i:\Theta^n(k) \to A^{n+1}.$
According to formula (\ref{decom}) we have
$
[\Theta^n(k)] = c_{(n)}^ \nu(\Theta^n(k)) \frac{[\Theta^n]}{(n+1)!},   
$
so we need only to compute $c_{(n)}^ \nu(\Theta^n(k)).$

The characteristic class  $c_{(n)}(\nu \Theta^n(k))=c_1^n(\mathcal L^{\otimes k}) =i^*c_1^n(L^{\otimes k})=k^n i^*c_1^n(L),$
so we have $c_{(n)}^ \nu(\Theta^n(k))= (c_{(n)}(\nu \Theta^n(k)), \langle\Theta^n(k)\rangle) =
k^n ( i^*c_1^n(L), \langle\Theta^n(k)\rangle)\\
= k^n( c_1^n(L), i_*\langle\Theta^n(k)\rangle) =
k^{n+1}( c_1^{n+1}(L), \langle A^{n+1}\rangle) = k^{n+1} (n+1)!,$ which together with (\ref{psik2}) proves the claim.
\end{proof}

%
%

%


\section{Real-analytic elliptic representatives}

Consider now the most degenerate case of the abelian variety, when $A^{n+1}$ is the product of $n+1$ copies of an elliptic curve $\mathcal E=\mathbb C/\Gamma,$ where $\Gamma$ is the lattice with periods $2\omega_1, 2\omega_2$ (in the classical notations from \cite{WW}).

Let $L$ be the canonical line bundle on it with the holomorphic section 
$$
S_0(u)=\sigma(u_1)\dots \sigma(u_{n+1}), \quad u=(u_1,\dots,u_{n+1})\in \mathcal E^{n+1},
$$
where $\sigma$ is the classical Weierstrass sigma function \cite{WW}.
The corresponding theta divisor defined by $S_0(u)=0$ is the union of $n+1$ coordinate hypersurfaces given by $u_i=0.$

 It is singular, but we can show now that one can find a smooth real-analytic representative of the same homology class, which can be determined in terms of classical elliptic functions.

Consider the classical Weierstrass zeta function $\zeta(z)$
 with simple poles at the lattice points and the transformation properties
 $$
 \zeta(z+2\omega_1)=\zeta(z)+2\eta_1, \quad  \zeta(z+2\omega_2)=\zeta(z)+2\eta_2,
 $$
 where
 $
 \eta_i=\zeta(\omega_i), \,\, i=1,2,
 $
see  \cite{WW}.
 
Introduce the following non-holomorphic function (inspired by the theory of periodic vortices \cite{HV,Tkac})
\begin{equation}
\label{xi}
\xi(z)=\zeta(z)+az+b\bar z,
\end{equation}
where $a,b$ is the unique solution of the following linear system
\begin{equation}
\label{lin}
a\omega_1+b \overline{\omega}_1+\eta_1=0,  \quad a\omega_2+b{\overline\omega}_2+\eta_2=0,
\end{equation}
or, explicitly
$$
a=-\frac{\eta_1 \overline{\omega}_2-\eta_2\overline{\omega}_1}{\omega_1 \overline{\omega}_2-\omega_2 \overline{\omega}_1}=-\frac{\eta_1\overline{\omega}_2-\eta_2\overline{\omega}_1} {2 \Im\,(\omega_1\overline{\omega}_2)}, \quad 
b=\frac{\eta_1\omega_2-\eta_2\omega_1}{\omega_1 \overline{\omega}_2-\omega_2 \overline{\omega}_1}=\frac{\pi}{4 \Im\,(\omega_1\overline{\omega}_2)},
$$
where we have used the Legendre identity \cite{WW}
$$
\eta_1\omega_2-\eta_2\omega_1=\frac{\pi i}{2}.
$$
\begin{Lemma}
The function $\xi$ is an odd doubly-periodic, complex-valued, real-analytic, harmonic function with the asymptotic behaviour
$\xi(z)= 1/z+\mathcal O(z)$ at zero and with zeros at all 3 half periods $\omega_1, \omega_2, \omega_3=\omega_1+\omega_3.$

In the lemniscatic case with $\omega_1=\omega \in \mathbb R, \, \omega_2=i\omega$ we have
\begin{equation}
\label{lem}
\xi(z)=\zeta(z)-\frac{\pi}{4\omega^2}\bar z,
\end{equation} 
which has zeros precisely at the 3 half-periods.
\end{Lemma}

\begin{proof} The property $\xi(-z)=-\xi(z)$ follows from the same property of $\zeta(z).$ 
The double-periodicity follows from the transformation properties of function $\zeta:$
$$
 \xi(z+2\omega_1)=\xi(z)+2\eta_1+2a\omega_1+2b\overline{\omega}_1=\xi(z),
 $$
 $$   
 \xi(z+2\omega_2)=\xi(z)+2\eta_2+2a\omega_2+2b\overline{\omega}_2=\xi(z)
 $$
 due to (\ref{lin}).  Combining these two properties we have $\xi(\omega_i)=-\xi(-\omega_i)=-\xi(\omega_i)$, so $\xi(\omega_i)=0$ for all half-periods.

In the lemniscatic case with $\omega_1=\omega \in \mathbb R, \, \omega_2=i\omega$ we have 
$\zeta(iz)=-i\zeta(z)$ and thus $\eta_2=-i\eta_1, \, \eta_1 \in \mathbb R.$
Corresponding $\wp$-function satisfies equation
$$
(\wp')^2=4\wp(\wp^2-e^2), \quad
e=\wp(\omega)=\frac{\Gamma^4(1/4)}{32\pi\omega^2},
$$
where $\Gamma$ is the classical Euler's $\Gamma$-function.

From Legendre identity we have
$
i\omega\eta_1-\omega\eta_2=2i\omega\eta_1=\frac{\pi i}{2},
$
so in this case
$$
\eta_1=\frac{\pi}{4\omega}, \,\, \eta_2=-i\frac{\pi}{4\omega},
$$
which gives
$
a=0, \,\, b=-\frac{\pi}{4\omega^2}
$
and the relation (\ref{lem}).

Note that the function $\xi(z)$ satisfies the equation
$
\partial \bar{\partial} \xi(z,\bar z)=0
$
and thus is indeed a complex-valued harmonic function. The zeros of such functions were extensively studied, see \cite{Duren} and references therein.

The number of the zeros of such functions depends on the position of zero in relation with the caustic defined as the image $\xi(\Sigma) \subset \mathbb C$ of the critical set
$$\Sigma:=\{z\in \mathbb C: J_\xi(z,\bar z)=0\},$$ where $J_\xi$ is the Jacobian of the map $\xi:  \mathbb R^2 \to \mathbb R^2.$ The real Jacobian of the harmonic function
$f(z,\bar z)=g(z)+h(\bar z)$ with holomorphic $g,h$ is
$$
J_f(z,\bar z)=|g'(z)|^2-|h'(z)|^2.
$$
Thus in our case the critical set is
$$\Sigma:=\{z\in \mathbb C: |\wp(z)|=\frac{\pi}{4\omega^2}\}.$$ 
Note that this level of the real function $F(z,\bar z)=|\wp(z)|$ is non-singular. Indeed,
if $\wp'(z)=0$ then $z$ is a half-period and thus $\wp(z)=0$, or $\wp(z)=\pm e.$
Since $3.6<\Gamma(1/4)< 3.7$ we have $$e=|\wp(z)|=\frac{\Gamma^4(1/4)}{32\pi\omega^2}>\frac{\pi}{4\omega^2}.$$
Thus by the general theory \cite{Duren} the equation $\xi(z)=c$ has one solution if $c$ lies outside the caustic and 3 solutions if $c$ is inside the caustic.
Since $c=0$ is inside the equation $\xi(z)=0$ has 3 solutions, which are precisely the half-periods.

Note that the corresponding Jacobians at the half-periods $\omega_1=\omega, \, \omega_2=i\omega$ are
$$
J(z,\bar z) =|\wp(\omega)|^2-\left(\frac{\pi}{4\omega^2}\right)^2=\left(\frac{\Gamma^4(1/4)}{32\pi\omega^2}\right)-\left(\frac{\pi}{4\omega^2}\right)^2>0,
$$
while at $\omega_3=\omega+i\omega$ we have
$$
J(z,\bar z)=|\wp(\omega_3)|^2-\left(\frac{\pi}{4\omega^2}\right)^2=-\left(\frac{\pi}{4\omega^2}\right)^2<0,
$$
so the sum of the indices is 1, as it is expected since the degree of the map $\xi: \mathcal E\to \mathbb CP^1$ is 1.
\end{proof}
 
Let $I \subset [n+1]:=\{1,2,\dots,n+1\}$ be a finite subset and denote
$$
\xi_I(u):=\prod_{i \in I} \xi(u_i), \quad u\in \mathcal E^{n+1},
$$
where for empty subset $I=\emptyset$ we set $\xi_I(u)\equiv 0.$

 Consider now the following family of the non-holomorphic (but real-analytic) sections of the complex line bundle $L$ given by
 \begin{equation}
\label{S}
S(u, a)=S_0(u)+S_0(u)\sum_{I,J \subset [n+1], I\cap J=\emptyset}a_{IJ}(\xi_I(u)+\xi_J(u))
\end{equation}
with arbitrary coefficients $a_{IJ}=a_{JI}\in \mathbb C,$ assuming that both subsets $I$ and $J$ cannot be empty simultaneously.

\begin{Theorem}
For generic coefficients $a_{IJ}$ the zero locus of this section $$\mathcal M_W(a)=\{u\in \mathcal E^{n+1}: S(u,a)=0\}\subset \mathcal E^{n+1}$$
 is a smooth connected real-analytic $U$-manifold, which can be used as a representative of the cobordism class of the coefficient $[\mathcal B^n]$ in the Chern-Dold character.
\end{Theorem}

\begin{proof}
Consider the set $$\mathcal M=\{(u,a): u\in \mathcal E^{n+1}, \, a=(a_{IJ}) \in \mathbb C^N, S(u, a)=0\} \subset \mathcal E^{n+1}\times \mathbb C^N.$$  We claim that this is a smooth submanifold of this product. Indeed, assume that
$$
\frac{\partial S}{\partial a_{IJ}}=S_0(u)(\xi_I(u)+\xi_J(u))=0
$$
for all pairs of non-intersecting subsets $I,J \subset [n+1].$
Then, in particular, we have
$$
\prod_{i=1}^{n+1}\sigma(u_i)=0, \,\, \prod_{i=1}^{n+1}\sigma(u_i)\xi(u_i)=0,
$$
so some of the coordinates of the potential singularities equal to 0 and some to a half-period.
Let $$I=\{i \in [n+1]: \sigma(u_i)=0\}, \, J=\{j \in [n+1]: \xi(u_i)=0\},\,I\cap J=\emptyset,$$
then the corresponding $\frac{\partial S}{\partial a_{IJ}}=S_0(u)(\xi_I(u)+\xi_J(u))=S_0(u)\xi_I(u)\neq 0,$ since $\sigma(z)\xi(z)=0$ only at half-periods. The contradiction means that the submanifold $\mathcal M$ is indeed non-singular.
Now the smoothness of $\mathcal M_W(a)$ follows from Sard's Lemma, saying that the set of critical values of the natural projection $\pi: \mathcal M \to \mathbb C^N$ has measure zero.

Since both $\mathcal M_W(a)$ and $\Theta^n$ are the smooth zero loci of two sections of the same line bundle, they have the same cobordism class. Now the claim follows from Theorem 1.1.
\end{proof}

Note that Theorem 4.9 implies that any multiple of the theta divisor $k^{n+1}[\Theta^n]=[\Theta^n(k)], \, k\geq 2$ does admit an explicit smooth algebraic realisation inside $\mathcal E^{n+1}$ as zero locus of a generic section of the line bundle $L^k.$ 

For example, for $k=2$ 
the space of sections of $L^2$ has dimension $2^{n+1}$ and is generated by
$$
\Phi_\varepsilon(u)=\prod_{k=1}^{n+1}\phi_{\varepsilon_k}(u_k), \,\, u=(u_1,\dots,u_{n+1})\in \mathcal E^{n+1}, \,\, \varepsilon= (\varepsilon_1,\dots, \varepsilon_{n+1})\in \mathbb Z_2^{n+1},
$$
\beq{Phieps2}
\phi_0(z)=\sigma^2(z), \quad \phi_1(z)=\sigma^2(z+\omega)e^{-2\eta z},  
\eeq
where $\omega$ is any half-period and $\eta=\zeta(\omega).$ Indeed, $\sigma(z)$ satisfies the following transformation properties under the shift by the periods \cite{WW}
$$
\sigma(z+2\omega_k)=-\sigma(z)e^{2\eta_k(z+\omega_k)},\,\, k=1,2
$$
which imply the transformation properties 
$$
\phi_\epsilon(z+2\omega_k)=\phi_\epsilon(z)e^{4\eta_k(z+\omega_k)}, \, \epsilon=0,1, \,\, k=1,2.
$$
For $\phi_0(z)=\sigma^2(z)$ this is straightforward, while for $\phi_1(z)$ this follows from the Legendre identity. From the results of the previous section we obtain

\begin{prop}
For generic coefficients $c_{\varepsilon}, \, \varepsilon \in \mathbb Z_2^{n+1}$ the zero locus
\beq{sec2}
\Phi_c(u):=\sum_{\varepsilon \in \mathbb Z_2^{n+1}}c_{\varepsilon}\Phi_\varepsilon(u)=0, \quad u=(u_1,\dots,u_{n+1})\in \mathcal E^{n+1}
\eeq
is a smooth irreducible algebraic variety, representing the cobordism class $2^{n+1}[\Theta^n]=[\Theta^n(2)].$
\end{prop}

\section{Discussion: Milnor-Hirzebruch problem}

The Milnor-Hirzebruch problem was first posed by Hirzebruch in his ICM-1958 talk \cite{Hirz}. Its algebraic version can be formulated in our notations as follows:

{\it Which sets of $p(n)$ integers $c_\lambda, \, \lambda \in \mathcal P_n$ can be realised as the Chern numbers $c_\lambda(M^n)$ of some smooth irreducible complex algebraic variety $M^n$?}

In this version it still remains largely open, although some arithmetic restrictions are known since the work of Milnor and Hirzebruch.

In particular, in the (complex) dimension $n=1,2,3$ we have the following congruences for the usual Chern numbers of any almost complex manifold (see Hirzebruch \cite{Hirz}, section 7):
$$
n=1: \quad c_1\equiv  0\mod 2, \quad\quad
n=2: \quad c_2+c_1^2 \equiv 0\mod 12,
$$
$$
n=3: \quad c_1c_2 \equiv 0\mod 24,\quad c_3 \equiv c_1^3\equiv 0\mod 2.
$$
$$
n=4: \quad -c_4+c_1c_3+3c_2^2+4c_1^2c_2-c_1^4 \equiv 0\mod 720,\quad c_1^2 c_2+2c_1^4\equiv 0\mod 12,
$$
$$
-2c_4+c_1c_3\equiv 0\mod 4.
$$
Since the total Chern class of $\Theta^n$ satisfies the relation (\ref{rel}) with $\mathcal D^n=(n+1)!$ all the characteristic numbers of $\Theta^n$ equal $\pm (n+1)!.$ 

In particular, for $n=1$ we have $c_1=-2$, for $n=2: c_1^2=c_2=6,$ for $$
n=3: c_1^3=-c_1c_2=c_3=-24,$$ $$n=4: c_1^4=c_1^2c_2=c_1c_3=c_2^2=c_4=120.$$

We see that the first Hirzebruch congruence in each case is sharp for the theta divisors, which means that it cannot be improved in the algebraic setting. 
This is related to the fact that the Todd genus $Td(\Theta^n)=(-1)^n.$

Recall that the divisibility conditions in terms of characteristic classes in $K$-theory were described by Hattori \cite{Hat-66} and Stong \cite{Stong-65, Stong-68}.

We can get now all divisibility conditions on the Chern numbers $c^{\nu}_\lambda(M^{2n})$ of $U$-manifolds in the following more effective way.
We apply the Landweber-Novikov operations $S_\lambda, \, |\lambda|<n$ to our formula (\ref{decom}) 
$$[M^{2n}]=\sum_{\lambda: |\lambda|=n}c^{\nu}_\lambda(M^{2n})\frac{[\Theta^\lambda]}{(\lambda+1)!}$$
and use that, according to Theorem 1.2, $S_\lambda[\Theta^n]=0$ unless $\lambda=(k), \, k < n$ when $S_{(k)}[\Theta^n]=[\Theta_{k}^{n-k}],$ which is a polynomial of $t_j=[\Theta^j], j=1,\dots, n-k$ with positive integer coefficients.
Substituting now $t_j=(-1)^j$ and demanding the result to be integer, we get all divisibility conditions on the Chern numbers $c^{\nu}_\lambda(M^{2n})$ of $U$-manifolds. 

 It is natural to ask if they can be improved for irreducible algebraic varieties.
We plan to discuss this in the light of our results elsewhere.

To conclude we would like to mention an important development due to Kotschik \cite{Kotschik}, who answered another question of Hirzebruch about all topologically invariant linear combinations of Chern numbers of smooth complex projective varieties (see also \cite{KS} for more results in this direction). 

\section{Acknowledgements}

We are very grateful to S.P. Novikov for his interest and encouragement, and to I. Cheltsov, S. Grushevsky, A. Prendergast-Smith and Yu. Prokhorov for very useful discussions of algebro-geometric aspects of this work. 

We are also grateful to W. Browder, J. Morava, T. Panov and anonymous referees for very helpful comments, helping to improve the earlier version of the paper.


\end{document}